\documentclass[]{article}
\usepackage{amsmath,amsthm}
\usepackage{amsfonts}
\usepackage{amssymb}

\theoremstyle{definition}
\newtheorem{defn}{Definition}[section]

\theoremstyle{definition}
\newtheorem{lemma}{Lemma}[section]

\theoremstyle{definition}
\newtheorem{theorem}{Theorem}[section]

\theoremstyle{definition}
\newtheorem{corollary}{Corollary}[section]

\theoremstyle{definition}

\theoremstyle{definition}
\newtheorem{obsv}{Observation}[section]

\theoremstyle{definition}
\newtheorem{remark}{Remark}[section]

\theoremstyle{definition}

\title{A new viewpoint of the G\"odel's incompleteness theorem and it's applications}
\author{Tianheng Tsui}

\begin{document}

\maketitle

\begin{abstract}
A new viewpoint of the G\"odel's incompleteness theorem be given in this article which reveals the deep relationship between the logic and computation. Upon the results of these studies, an algorithm be given which shows how to search a proof of statement in first order logic from finite concrete examples, and an approach be proposed to improve searching mathematical proof by neural network.
 
\end{abstract}

\begin{center}
\section{Informal Introduction}
\end{center}

G\"odel's incompleteness theorem are the most famous result in modern logic. In addition to the G\"odel's original proof, there are some other proofs of this theorem \cite{wpd}. In this article, a new viewpoint to interpret the G\"odel's incompleteness theorem will be given, which lead to some interesting applications and a deeper understanding of the relation between logic and computation. 

The notion of ``proof'' plays a central role in mathematics as the means by which the truth or falsity of mathematical statements is established. The main difficulty in mathematical ``proof'' is how to prove ``infinte objects'' serve the given  mathematical statement. If a mathematical statement talk about finite objects, for example, the statement: $\forall n \in N ( (0<n<100)\to (n^2<10000) )$, we can test it one by one on the finite objects to prove or disprove the statement. That means there are only finite cases have to be tested for the statement. But if a mathematical statement talk about infinite objects, for example, the statement: $\forall n ( (n>2)\to (n^2>2n) )$, we cannot  prove it or disprove it by test it one by one, but we have to classify all the infinite objects into essentially finite different cases, in each case there is a corresponding independent reason for the statement to be true or to be not true, thus we can prove or disprove the mathematical proposition about infinite objects. So we get the not rigorous but heuristic intuition observatons:
\begin{center}
\end{center}
\begin{obsv}
A true mathematical statement can be proved within finite steps in a consistent effective formal system if and only if, the domain of the statement can be split into essentially finite different classes, in each class there is a corresponding independent reason to govern the members to serve the statement, i.e., there are only essentially finite different independent reasons to make the statement to be true. 
\end{obsv}

Let a statement $\varphi$ be true in natrual numbers and unprovable in the formal Peano system. So just within the formal Peano system, only we can do is to test it one by one on the standard natrual number $0,1,2,3,\cdots$. if $\varphi (0)$ is true, we get a true sample, if $\varphi (1)$ is true, we get another true sample, \ldots. The reason to make a natrual number $m$ satisfys $\varphi$ may be different and independent from other numbers'. There are infinite different reasons to make the statement $\varphi$ to be true. Since essentially each number $m$ has a particular reason to make the statement $\varphi (m)$ to be true and $\varphi$ is not logical consequence of Peano Arithmetic, before it is tested: ``computed'' the $\varphi (m)$ and ``abserved'' the result in the Peano system, nobody can predict true or false of $\varphi (m)$ just within Peano system. Thus when we check the value of the true but unprovable statement $\varphi$ on the natrual numbers, the results seem that you are tossing a coin and  every time the result is up! Although it is not logically impossible, it's probablilty is $0$ from prpbability theory. So we get the second intuition observatons:

\begin{obsv}
	\label{obsv2}
	Let $\mathcal{L}$ be a language, $T$ is an $\mathcal{L}-theory$, $\mathcal{M}$ is a model of $T$.
	A mathematical statement $\varphi$ is true in $\mathcal{M}$, expressible in $\mathcal{L}$, but it cannot be proved from $T$, if and only if, the domain of the statement in the model $\mathcal{M}$ cannot be split into  essentially finite different classes, in each class there is a corresponding independent reason by $T$, which is expressible in $\mathcal{L}$, govern the members to serve the statement, i.e., there are essentially infinite different independent reasons govern the whole domain to serve the unprovable true statement, but any finite reasons do not. 
\end{obsv}

The rigorous expression of the \textbf{observation} \ref{obsv2} is the \textbf{Theorem} \ref{keytheorem}.

\begin{center}
	\item
	\section{Preliminaries and Notations}
\end{center}

This section is devoted to the exposition of basic preliminary material, notations and conventions which be  used throughout of this article.

The notion of algorithm can be defined in terms of Turing machines by means of the Church–Turing thesis, so the sentence ``There is an algorithm \ldots '' means ``There is a Turing machine to compute \ldots '', and sometimes Turing machine algorithms be described in very high level. If a function or a map is recursive, it means that the function or the map can be computed by a Turing machine.

It is well known that there are character encoding system \textbf{ASCII} and language encoding system \textbf{\LaTeXe}. Therefore, throughout this article, we assume all the mathematical objects be encoded by these fixed encoding systems, and the length of a mathematical object is the number binary bits to represent the object. For example, the symbols ``t'', ``$t$'' and ``$t_{298}$'' is represented as one character: ``t'', three characters: ``\$t\$'' and nine characters ``\$t\_\{298\}\$'' in \textbf{\LaTeXe}, each charater be encoded by seven bits in \textbf{ASCII}, therefore the binary length of these objects are 7, 21 and 63 respectively. 

It should be noticed that the symbol ``$ t_{298}   $'' can be represented as ``\$t\_\{298\}\$'' and ``\$  t\_\{298\}\ \ \ \$'', we take the shortest representation to calculate its length.\\

\begin{defn}
	\label{defn_string_len}
	Let $s$ is a \textbf{ASCII} string, the ASCII length of $s$, written $\textbf{asciilen}(s)$, abbreviated $\|s\|_{as}$, is the number of characters that it contains,  and $$\textbf{binlen}(s) = \textbf{asciilen}(s) \times 7 $$ named binary length of $s$.
\end{defn}

Let the formal Zermelo-Fraenkel axiomatic set theory is denoted by ZF, and ZFC denotes the theory ZF with the Axiom of Choice, and $\mathcal{\omega}$ represents the natural number set $N$ in the formal ZFC system.

\begin{defn}
	\label{defn_classical_tm}
	A Turing machine $\textbf{M}$ is a 5-tuple, $(Q,\Gamma,\delta,q_0,q_{halt})$
	
	$Q$ is a finite set of states, i.e., $\exists i(i \in \omega \wedge \Vert Q \Vert = i)$,
	
	$\Gamma$ is the tape alphabet containing the blank symbol $\sqcup$, and the left end symbol $\triangleright$,
	
	$\delta$:$Q \times \Gamma \longrightarrow Q \times \Gamma \times \{L,S,R\}$ is the transition function,
	
	if $\delta(q,\triangleright) = (p,s,b)$, then $(s = \triangleright) \wedge (b = R)$,
	
	if $\delta(q, a) = (p,s,b)$ and $a \ne \triangleright$, then $s \ne \triangleright $,
	
	$q_0 \in Q$ is the start state,
	
	$q_{halt} \in Q$ is the halt state, that is $\forall a \in \Gamma$:
	
	$\delta(q_{halt}, a) = (q_{halt},a,S)$, and
	
	$\forall q \in Q (\forall a \in \Gamma (\delta(q, a) = \delta(q_{halt}, a)) \rightarrow (q=q_{halt}) )$.

\end{defn}

	Unless otherwise indicated, it will always be assumed that the tape alphabet $\Gamma = \{0,1, \sqcup, \triangleright \}$ throughout this article, and we assume the basic notions and results of mathematical logic, such as formula, sentence, the set of all formulas is recursive \ldots, etc.

%There are some assumptions in this article:

\begin{defn}
	\label{time_complexity}
	(\textbf{time complexity})  Let $M$ be a Turing machine that halts on all inputs. The running time or time complexity of $M$, denoted by $t_M$, is the function $$t_M:\ \omega \rightarrow \omega$$ where $t_M(n)$ is the maximum number of steps that $M$ uses on any input of length $n$. 
\end{defn}

In computational complexity theory, a reasonable assumption $t_M(n) \ge n $  is to allow the algorithm have time to read its input. But in this section a property of  the machine $M$ with running time $t_M(n) < n $ be given, and the relationship between it and the provability of statement in consistent effective formal system will be revealed in later.

\begin{theorem}
	\label{short_time_tm}
	Let $M$ be a Turing machine, the length of input string $s$ be denoted as $\|s\|$. If there exist a number $K$ for any input $s$,  $$ \|s\| \ge K \rightarrow t_M(\|s\|) <  \|s\| $$ then  $\forall r ( \|r\|\ge K) \rightarrow (t_M(r)<K)$ and if $(\|s\| \le K \rightarrow M(s) = 1)$ then we can prove $\forall s M(s) = 1$ in ZFC.
\end{theorem}

\begin{proof}
	Let $\|s\| = K$, from the assumption $t_M(\|s\|) <  \|s\|$, the machine $M$ halts before it reads the last bit of the input $s$ i.e., it never reach to the end boundary of the input, the bits following the $(K-1)$th bit have no effect on computation.

	Therefore  if input $\|r\| \ge K$, and the first $K-1$ bits are the same as a string $s$ with $\|s\|=K$, the machine does not discriminate $r$ from  $s$, when computing on $r$, it return the same result as computing on $s$, and it halts after the same steps, i.e., $M(r)=M(s)$ and $t_M(\|r\|) = t_M(\|s\|) <  \|s\|  =K$.
	
	The number of string $s$ with $\|s\| \le K$ is finite. From the explaination above and if $(\|s\| \le K \rightarrow M(s) = 1)$,  obviously  we can prove  $\forall s M(s) = 1$ in finite steps in ZFC.

\end{proof}

\begin{defn}
	\label{frmsq}
	Let the set of all formulas is denoted by $\textbf{Frm}$, and let $\textbf{Frmsq}$ denotes the set of all finite formula sequences, i.e., $sq \in \textbf{Frmsq}$ if and only if $sq$ is a finite formula sequence: $$sq = \langle s_0,s_1,\ldots,s_r \rangle, r\in \omega$$
	For more rigorous, $sq$ is a map from $r+1$ to $\textbf{Frm}$ such that $$sq(i) \in \textbf{Frm},\ \forall i<r+1.$$

\end{defn}

\begin{defn}
	\label{rtheory}
	$\textbf{RTheory} = \{ \langle T, al \rangle |\ T \subseteq \textbf{Frm} \text{ and }  al \text{ is an algorithm  which} $\\$\text{decide whether }  \varphi \in T, \text{ for any formula } \varphi \text{ i.e., } T \text{ is recursive} \} $. If $T$  is a finite set, we assume that $T$ is a formula sequence: $\langle {\varphi}_0, {\varphi}_1, {\varphi}_2, \ldots, {\varphi}_n \rangle$ i.e., $T \in \textbf{Frmsq}$.
\end{defn}

\begin{defn}
	The set $\Lambda$ of logical axioms are arranged in seven groups:\\
	
	\begin{enumerate}
		\item Tautologies;
		\item $\forall x \alpha \rightarrow \alpha_t^x $, where $t$ is substitutable for $x$ in $\alpha$;
		\item $\forall x (\alpha \rightarrow \beta) \rightarrow (\forall x \alpha \rightarrow \forall x \beta)$;
		\item $\alpha \rightarrow \forall x \alpha$, where $x$ does not occur free in $\alpha$;
		\item $x=x$;
		\item $(x=y) \rightarrow (\alpha \rightarrow \beta)$, if $\alpha$ and $\beta$ are atomic formulas and $\beta$ is obtained from $\alpha$ by replacing an occurrence of $x$ in $\alpha$ by $y$;
		\item $\alpha_t^x \rightarrow \exists x \alpha$, where $t$ is substitutable for $x$ in $\alpha$.
	\end{enumerate}
\end{defn}

\begin{defn}
	\label{compute_checker}
	Let $\langle T, al \rangle \in \textbf{RTheory}$, a proof $\pi$ of a statement $\varphi$ from $T$ in ZFC is a finite sequence  $\langle {\varphi}_0, {\varphi}_1, {\varphi}_2, \ldots, {\varphi}_n \rangle $ of formulas such that ${\varphi}_n$ is ${\varphi}$ and for each $i \leq n$ one of the following conditions holds: 
	
	\begin{enumerate}
		\item ${\varphi}_i \in \Lambda$;  
		\item ${\varphi}_i \in \text{ZFC}$;
		\item ${\varphi}_i \in T$;
		\item $\exists j,k < i$ such that ${\varphi}_j = {\varphi}_k \rightarrow {\varphi}_i$;
		\item $\exists j < i \exists k \in \omega ({\varphi}_i = \forall x_k {\varphi}_j) $.
	\end{enumerate}
	
	and denoted as
	$$\langle T, al \rangle {\vdash}^{\pi} \varphi$$
	
	From the \textbf{definition} \ref{frmsq}, ${\pi}(i) = {\varphi}_i $, and it is easy to see that there are algorithms decide the corresponding conditions such as:
	
	\begin{enumerate}
		\item $\textbf{G}_{\Lambda}({\pi}(i)) = 1 \text{ if } {\pi}(i) \in \Lambda, \text{ otherwise, 0} $.
		\item $\textbf{G}_{\text{ZFC}}({\pi}(i)) = 1 \text{ if } {\pi}(i) \in \text{ZFC}, \text{ otherwise, 0} $.
		\item $\textbf{G}_{\text{IN}}(\langle T,al \rangle ,{\pi}(i)) = 1 \text{ if } {\pi}(i) \in T, \text{ otherwise, 0} $, note that $\textbf{G}_{\text{IN}}$ use the algorithm $al$ to decide whether ${\pi}(i) \in T$.
		\item $\textbf{G}_{\rightarrow }(\pi,i,j,k) = 1 \text{ if } (j,k<i) \text{ and } {\pi}(j) = {\pi}(k) \rightarrow {\pi}(i), \text{ otherwise, 0} $.
		\item $\textbf{G}_{\forall}(\pi,i,j,k) = 1 \text{ if }  j < i \text{ and }  k \in \omega ({\pi}(i) = \forall x_k {\pi}(j)), \text{ otherwise, 0} $.
		
	\end{enumerate}
	Let the set of above five verification algorithms is $$\textbf{G}_{\text{check}} = \{ \textbf{G}_{\Lambda}, \textbf{G}_{\text{ZFC}}, \textbf{G}_{\text{IN}}, \textbf{G}_{\rightarrow }, \textbf{G}_{\forall} \}$$
	and its member is called checker.

\end{defn}

\begin{defn}
	\label{proof type}
	(\textbf{proof type}) Let $\langle T, al \rangle \in \textbf{RTheory} \text{ and } \pi \in \textbf{Frmsq},\ \pi = \langle {\varphi}_0, {\varphi}_1, {\varphi}_2, \ldots, {\varphi}_n \rangle$, the proof type of $\langle \langle T, al \rangle, \pi \rangle$ denoted by $$\textbf{prooftype}(\langle T, al \rangle, \pi)$$ such that: If \textbf{not} $\langle T, al \rangle {\vdash}^{\pi}  {\varphi}_n$, 
	$$\textbf{prooftype}(\langle T, al \rangle, \pi) \text{ is the empty set}.$$
	else if $\langle T, al \rangle {\vdash}^{\pi}  {\varphi}_n$, then $\textbf{prooftype}(\langle T, al \rangle, \pi)$ is also called the proof type of $\langle T, al \rangle {\vdash}^{\pi}  {\varphi}_n$, and it is a same length sequence $G = \langle g_0, g_1,g_2 \ldots, g_n \rangle$ of checkers, such that:
	
	\begin{enumerate}
		\item If ${\varphi}_i \in \Lambda$, then the corresponding $g_i$ is a recursive function on $\textbf{Frmsq}$ such that $\forall sq \in \textbf{Frmsq}$: $$g_i(sq) = \textbf{G}_{\Lambda}(sq(i))$$
		Say that $g_i$ is a $\textbf{G}_{\Lambda}$ type checker.
		\item If ${\varphi}_i \in \text{ZFC}$, then the corresponding $g_i$ is a recursive function on $\textbf{Frmsq}$  such that $\forall sq \in \textbf{Frmsq}$: $$g_i(sq) = \textbf{G}_{\text{ZFC}}(sq(i))$$
		Say that $g_i$ is a $\textbf{G}_{\text{ZFC}}$ type checker.
		\item If ${\varphi}_i \in T$, then the corresponding $g_i$ is a recursive function on $\textbf{RTheory} \times \textbf{Frmsq} $  such that $\forall sq \in \textbf{Frmsq}$ and $S = \langle ST, al_s \rangle \in \textbf{RTheory}$: $$g_i(S,sq) = \textbf{G}_{\text{IN}}(S, sq(i))$$ 
		Say that $g_i$ is a $\textbf{G}_{\text{IN}}$ type checker.
		\item If $\exists j,k < i$ such that ${\varphi}_j = {\varphi}_k \rightarrow {\varphi}_i$, then the corresponding $g_i$ is a recursive function on $\textbf{Frmsq}$ such that $\forall sq \in \textbf{Frmsq}$: $$g_i(sq) = \textbf{G}_{\rightarrow }(sq,i,j,k)$$
		Say that $g_i$ is a $\textbf{G}_{\rightarrow }$ type checker.
		\item If $\exists j < i \text{ and } k \in \omega ({\varphi}_i = \forall x_k {\varphi}_j) $, then the corresponding $g_i$ is a recursive function on $\textbf{Frmsq}$ such that $\forall sq \in \textbf{Frmsq}$: $$g_i(sq) = \textbf{G}_{\forall}(sq,i,j,k)$$
		Say that $g_i$ is a $\textbf{G}_{\forall}$ type checker.
	\end{enumerate}
	$G$ is also called the adjoint check sequence of  $\pi$.
	
\end{defn}
% adjoint check sequence 

\begin{defn}
	\label{reason_type}
	Let $\langle T_1, al_1 \rangle {\vdash}^{\pi_1} \varphi$, the adjoint check sequence of $\pi_1$ is $G_1$,  $\langle T_2, al_2 \rangle {\vdash}^{\pi_2} \phi$, the adjoint check sequence of $\pi_2$ is $G_2$. Say that the proof type of $\langle T_1, al_1 \rangle {\vdash}^{\pi_1} \varphi$ is the same as the proof type of $\langle T_2, al_2 \rangle {\vdash}^{\pi_2} \phi$ if $G_1=G_2$.
\end{defn}

\begin{defn}
	\label{adjoint_checker}
	(\textbf{adjoint checker}) It is not hard to see that $G = \langle g_0,g_1, g_2, \ldots, g_n \rangle$, the proof type of  $\langle T, al \rangle {\vdash}^{\pi} \varphi$, can be easily converted to an algorithm which decide whether a proof have the same type, denote the algorithm by $\textbf{CK}_{(\langle T, al \rangle, \pi)}$, abbreviated $\textbf{CK}_{\pi}$, and it is called the adjoint checker of $\langle T, al \rangle {\vdash}^{\pi} \varphi$, which on input $$( \langle U, b \rangle, \sigma )  ,\text{ where } \sigma =\{\phi_0,\phi_1,\phi_2,\ldots,\phi_m \}$$ it does:
	
	firstly, it compare $m$ to $n$, if $m \ne n$, return 0 and stop, else it does the following operations:\\
	
	for all $0 \le i \le n$ it compute $g_i$ such as:
	\begin{enumerate}
		\item if $g_i = \textbf{G}_{\Lambda}(sq(i))$, then it compute $\textbf{G}_{\Lambda}(\sigma(i))$;
		\item if $g_i = \textbf{G}_{\text{ZFC}}(sq(i))$, then it compute $\textbf{G}_{\text{ZFC}}(\sigma(i))$;
		\item if $g_i = \textbf{G}_{\text{IN}}(S, sq(i))$, then it compute $ \textbf{G}_{\text{IN}}(\langle U, b \rangle, \sigma(i))$;
		\item if $g_i = \textbf{G}_{\rightarrow }(sq,i,j,k)$, then it compute $\textbf{G}_{\rightarrow }(\sigma,i,j,k)$;
		\item if $g_i = \textbf{G}_{\forall}(sq,i,j,k)$, then it compute $\textbf{G}_{\forall}(\sigma,i,j,k)$.
	\end{enumerate}
	If all of the computations of $g_i,\ 0 \le i \le n$, return 1, the $\textbf{CK}_{(\langle T, al \rangle, \pi)}$ return 1 and stop, else return 0 and stop, therefore
	$$\textbf{CK}_{(\langle T, al \rangle, \pi)}(\langle U, b \rangle, \sigma)=\begin{cases}
	1, &\text{ if } \textbf{prooftype}(\langle T, al \rangle, \pi) = \textbf{prooftype}(\langle U, b \rangle, \sigma)\\
	0, &\text{otherwise}	
	\end{cases}
	$$
	Indeed, the algorithm $\textbf{CK}_{(\langle T, al \rangle, \pi)}$ is described by a group of checkers $g_i$, it is only depend on the proof sequence $\pi$, it is therefore abbreviated to $\textbf{CK}_{\pi}$.
\end{defn}

\begin{theorem}
	\label{self_check}
	$\textbf{CK}_{(\langle T, al \rangle, \pi)}((\langle T, al \rangle, \pi)=1$, i.e.,$$\text{If }\langle T, al \rangle {\vdash}^{\pi} \varphi \text{ then }\textbf{CK}_{\pi}((\langle T, al \rangle, \pi)=1$$
\end{theorem}

\begin{proof}
	It is obvious from the definition \ref{adjoint_checker}.
\end{proof}

\begin{center}
	\item
	\section{A theorem of provability and an algorithm of proof}
\end{center}

In order to prove a statement, we may enumerate formula sequences, and verify the sequences, one by one, whether or not  it is a proof sequence of the statement. But it is not a practical method.
In practice, mathematicians often have computed lots of concrete examples before proposing a conjecture by intuition, and searching a proof of it guided by intuition. In this section, the prove process will be studied from the computational viewpoint, and give a rigorous expression of the following statement: ``There are essentially infinite different independent reasons govern the whole domain to serve the unprovable true statement''(\textbf{Theorem} \ref{keytheorem}), and give an algorithm which explain some aspects of practical prove activities.\\

%Firstly, the basic preliminay material and background be given.  

%There is no ``tape'' in the classical formal definitin of Turing machine. The ``tape'' is, indeed, the measurement equipment of computation. For convenience, the new formal definition of Turing machine include the ``tape'' as alphabet series, which leads to a more suitable definition of the concept of ``accept''.

\begin{defn}
	\label{cpt}
	
	Let $M$ is a Turing machine: $(Q,\Gamma,\delta,q_0,q_{halt})$,
	
	$t$ is a computation tape square, or simply tape square, if $t \in \Gamma \times Q \times \{0,1\}$,
	
	%$T = \{t| t = (t_0,t_1,t_2,\cdots),(\forall i \in \omega\ t_i \in \Gamma \times \{0,1\})\wedge\ (t_0 \in \{ \triangleright \} \times \{0,1\} )\wedge(\forall i (i\neq 0 \rightarrow t_i \notin  \{ \triangleright \} \times \{0,1\})) \wedge ( \exists i \in \omega ( (t_i \in \Gamma \times \{1\}) ) \wedge (\forall j \in \omega (j \ne i) \rightarrow (t_j \in \Gamma \times \{0\}) ) )  \}$,
	
	$T_M = \{t| t = (t_0,t_1,t_2,\cdots),\forall i \in \omega\ t_i \in \Gamma \times Q \times \{0,1\} \}$,
	
	%if $t \in T$, we say $t$ is a tape configuration of the Turing machine $M$,
	
	$Table_M = \{table| table =(tape_0,tape_1,tape_2,\cdots), \text{where}\ \forall i \in \omega (tape_i \in T_M)   \}$.
	
	%if $tb \in Table$, we say $tb$ is a table of the Turing machine $M$.
	
	$\pi_M$ is a projection function from $\Gamma \times Q \times \{0,1\}$ to $\Gamma$ as: $\pi_M(s,p,v) = s$.

\end{defn}

\begin{remark}
	\label{ctbl}
	A table of the Turing machine $M$ can be describe as following figure.
	
	\begin{center}
		\setlength{\unitlength}{0.5cm}
		\begin{picture}(10,10)
			\put(0,9){\line(1,0){8}}
			\put(0,8){\line(1,0){8}}
			\put(0,7){\line(1,0){8}}
			\put(0,6){\line(1,0){8}}
			\put(0,5){\line(1,0){8}}

			\put(0,9){\line(0,-1){8}}
			\put(1,9){\line(0,-1){8}}
			\put(2,9){\line(0,-1){8}}
			\put(3,9){\line(0,-1){8}}
			\put(4,9){\line(0,-1){8}}

			\put(0.1,8.35){$t_{00}$}
			\put(1.1,8.35){$t_{01}$}
			\put(2.1,8.35){$t_{02}$}
			\put(3.1,8.35){$t_{03}$}
			\put(5.1,8.35){$\cdots$}

			\put(0.1,7.35){$t_{10}$}
			\put(1.1,7.35){$t_{11}$}
			\put(2.1,7.35){$t_{12}$}
			\put(3.1,7.35){$t_{13}$}
			\put(5.1,7.35){$\cdots$}

			\put(0.1,6.35){$t_{20}$}
			\put(1.1,6.35){$t_{21}$}
			\put(2.1,6.35){$t_{22}$}
			\put(3.1,6.35){$t_{23}$}
			\put(5.1,6.35){$\cdots$}

			\put(0.1,5.35){$t_{30}$}
			\put(1.1,5.35){$t_{31}$}
			\put(2.1,5.35){$t_{32}$}
			\put(3.1,5.35){$t_{33}$}
			\put(5.1,5.35){$\cdots$}
			
			\put(0.1,4.35){$\cdots$}
			\put(1.1,4.35){$\cdots$}
			\put(2.1,4.35){$\cdots$}
			\put(3.1,4.35){$\cdots$}
			\put(5.1,4.35){$\cdots$}

			\put(0.1,3.35){$\cdots$}
			\put(1.1,3.35){$\cdots$}
			\put(2.1,3.35){$\cdots$}
			\put(3.1,3.35){$\cdots$}
			\put(5.1,3.35){$\cdots$}

			\put(0.1,2.35){$\cdots$}
			\put(1.1,2.35){$\cdots$}
			\put(2.1,2.35){$\cdots$}
			\put(3.1,2.35){$\cdots$}
			\put(5.1,2.35){$\cdots$}

		\end{picture}
		
	\end{center}
	
	%The rows are tape configurations of the Turing machine $M$,  $t_{i,j} \in \Gamma \times Q \times \{0,1\} $, which satisfy some conditions
	
	The rows are tape configurations of the Turing machine $M$  which satisfy some conditions.

	In ZFC system, a tape configuration of the Turing machine $M$: $$t = (t_0,t_1,t_2,\cdots),(\forall i \in \omega\ t_i \in \Gamma \times Q \times \{0,1\})$$ can be defined as function from $\omega$ to $\Gamma \times Q \times \{0,1\}$,  and a table can be defined as function $tb$ from $\omega$ to $T_M$  satisfys extra conditions which will be shown later.
	
	 %$(\forall i \in \omega\ t(i) \in \Gamma \times \{0,1\})\wedge\ (t(0) \in \{ \triangleright \} \times \{0,1\} )\wedge(\forall i (i\neq 0 \rightarrow t(i) \notin  \{ \triangleright \} \times \{0,1\})) \wedge ( \exists i \in \omega ( (t(i) \in \Gamma \times \{1\}) ) \wedge (\forall j \in \omega (j \ne i) \rightarrow (t(j) \in \Gamma \times \{0\}) ) )  $,
	 
	 Therefore the $t_{i,j}$ as above figure can be represented as $t(i)(j)$ in ZFC formal system. For convenience later,  let ``$t_{i,j}$'' is the abbreviation of ``$t(i)(j)$'' i.e., $t$ is a function from $\omega$ to $T_M$, $t(i)$ is a function from $\omega$ to $\Gamma \times Q \times \{0,1\}$ and $t_{i,j} = t(i)(j) \in \Gamma \times Q \times \{0,1\}$, a computation tape square.
	 
	 If $t_{i,j} = t(i)(j) \in \Gamma \times Q \times \{1\}$, we say the machine $M$ is reading the $j$th square of the tape at step $i$.
\end{remark}

\begin{defn}
	\label{def_tm_and_input}
	It is easy to see that the relation $t \in Table_M \text{ and } \pi_M$ can be defined within finite formulas in ZFC. We say that $t$ is a table of $M$ computing on input $s$, if $t \in Table_M$ and satisfys extra conditions such as:
	
	\begin{itemize}
		\item  $\forall i  ( t_{i,0} \in \{ \triangleright \} \times Q \times \{0,1\}) $, this means the leftmost end of a tape is always markered by $\triangleright$.
		\item  $\forall i \exists k   (t_{i,k} \in \Gamma \times Q \times \{1\})  \wedge (\forall j  (j \ne k) \rightarrow (t_{i,j} \in \Gamma \times Q \times \{0\}) )$, this formula means that there is one and only one square be reading at any time by the machine.	
	\end{itemize}
	and some interpretations of the transition function $\delta$: $Q \times \Gamma \longrightarrow Q \times \Gamma \times \{L,S,R\}$ such as:
	
	\begin{itemize}
		\item if $\delta(q,a) = (p,b,L)$, then the corresponding formula is 
		
		$\forall i \forall j (t_{i,j} = (a, q, 1) \rightarrow ( t_{(i+1),j} = (b,p,0) )\wedge( t_{(i+1),(j-1)} = (\pi_{M}(t_{i,(j-1)}),p,1) )\wedge (\forall k (k \ne j \wedge k \ne j-1) \rightarrow t_{(i+1),k} = (\pi_{M}(t_{i,k}),p,0)) ) $.
		
		\item if $\delta(q,a) = (p,b,S)$, then the corresponding formula is 
		
		$\forall i \forall j (t_{i,j} = (a, q, 1) \rightarrow ( t_{(i+1),j} = (b,p,1) )\wedge (\forall k (k \ne j) \rightarrow t_{(i+1),k} = (\pi_{M}(t_{i,k}),p,0)) ) $.
		
		\item if $\delta(q,a) = (p,b,R)$, then the corresponding formula is 
		
		$\forall i \forall j (t_{i,j} = (a, q, 1) \rightarrow ( t_{(i+1),j} = (b,p,0) )\wedge( t_{(i+1),(j+1)} = (\pi_{M}(t_{i,(j+1)}),p,1) )\wedge (\forall k (k \ne j \wedge k \ne j+1) \rightarrow t_{(i+1),k} = (\pi_{M}(t_{i,k}),p,0)) ) $.
		\item \ldots
		\item \ldots
		\item \ldots
		\item $\exists m (\forall j  (j < m) \rightarrow t_{0,j} \ne (\sqcup, q_0, 0) ) \wedge (\forall j  (j \ge m) \rightarrow t_{0,j} = (\sqcup, q_0, 0) ) $,
	\end{itemize}
	Let the set of these formulas are arranged as formula sequence: $\langle{\varphi}_0,{\varphi}_1,{\varphi}_2,\ldots,{\varphi}_k \rangle$, and denoted by $\textbf{def}_M$.\\

	If $s = s_0s_1s_2\ldots s_l,\ s_i \in \{0,1\}, 0 \leq i \leq l$ then there are $l+3$ formulas to describe the input $s$ on $M$ such as:
	
	\begin{itemize}
		\item ${\varphi}_{k+1} $: $\forall j  (j > l+1) \rightarrow t_{0j} = (\sqcup, q_0, 0)$,
		\item ${\varphi}_{k+2} $: $ t_{0,0} = (\triangleright, q_0, 1)$,
		\item ${\varphi}_{k+3} $: $ t_{0,1} = (s_0, q_0, 0)$,
		\item ${\varphi}_{k+4} $: $ t_{0,2} = (s_1, q_0, 0)$,
		\item ${\varphi}_{k+5} $: $ t_{0,3} = (s_2, q_0, 0)$,
		\item \ldots,
		\item ${\varphi}_{k+l+3} $: $ t_{0,(l+1)} = (s_l, q_0, 0)$.
		
	\end{itemize}
	Let the set of these $l+3$ formulas be denoted by $\textbf{input}_s$.

	It is not hard to see that $\textbf{def}_M$ is unchanged if the machine $M$ is fixed, and if $\|s\| = l$ then $\textbf{input}_s$ have $l+3$ formulas.

	$\textbf{T}_{\langle M,s \rangle} $ is a formula sequence such that  $$\textbf{T}_{\langle M,s \rangle} = \langle \textbf{def}_M \cup \textbf{input}_s \rangle = \langle {\varphi}_{0},{\varphi}_{1},{\varphi}_{2},\ldots,{\varphi}_{k+l+3} \rangle  $$ where we arrange its members as: the first $k+1$ formulas belonging to $\textbf{def}_M$, the following $l+3$ formulas belonging to $\textbf{input}_s$, and keep the order as above described, i.e., 
	\begin{itemize}
		\item ${\varphi}_{0}$,
		\item ${\varphi}_{1}$,
		\item ${\varphi}_{2}$,
		\item \ldots,
		\item ${\varphi}_{k}   \text{ is } \exists m (\forall j  (j < m) \rightarrow t_{0,j} \ne (\sqcup, q_0, 0) ) \wedge (\forall j  (j \ge m) \rightarrow t_{0,j} = (\sqcup, q_0, 0) ) $,
		\item ${\varphi}_{k+1} \text{ is } \forall j  (j > l+1) \rightarrow t_{0,j} = (\sqcup, q_0, 0)$,
		\item ${\varphi}_{k+2} \text{ is } t_{0,0} = (\triangleright, q_0, 1)$,
		\item ${\varphi}_{k+3} \text{ is } t_{0,1} = (s_0, q_0, 0)$,
		\item ${\varphi}_{k+4} \text{ is } t_{0,2} = (s_1, q_0, 0)$,
		\item ${\varphi}_{k+5} \text{ is } t_{0,3} = (s_2, q_0, 0)$,
		\item \ldots,
		\item ${\varphi}_{k+l+3} \text{ is } t_{0,(l+1)} = (s_l, q_0, 0)$.
		
	\end{itemize}
	Thus $\textbf{T}_{\langle M,s \rangle}$ is a $k+l+4$ formulas sequence as above.
\end{defn}

Thus we can define an algorithm $\textbf{AL}$ whose input is $\langle \varphi,\textbf{T}_{\langle M,s \rangle} \rangle$ and decide whether $\varphi \in \textbf{T}_{\langle M,s \rangle}$, such as:

\begin{defn}
	\label{searching_al}
	(\textbf{AL})  The algorithm on input $\langle \varphi,\textbf{T}_{\langle M,s \rangle} \rangle$, it compare $\varphi$ to each ${\varphi}_{i} \text{ in } \textbf{T}_{\langle M,s \rangle} $ 
	
	\begin{itemize}
		\item it compare  $\varphi$ to ${\varphi}_{0}$, if they match, return yes and halt, else:
		\item it compare  $\varphi$ to ${\varphi}_{1}$, if they match, return yes and halt, else:
		\item it compare  $\varphi$ to ${\varphi}_{2}$, if they match, return yes and halt, else:
		\item \ldots
		\item it compare  $\varphi$ to ${\varphi}_{k+l+3}$, if they match, return yes and halt, else return no and halt.
	\end{itemize}
	denoted such algorithm by $\textbf{AL}$ throughout this article.
\end{defn}

\begin{defn}
	\label{normal_proof}
	(\textbf{normal proof})  Let $ M $ is a Turing machine $(Q,\Gamma,\delta,q_0,q_{halt})$, $v \in \{0,1\}, s \in \{0,1\}^*$, and \textbf{AL} is the algorithm as the \textbf{definition} \ref{searching_al}. A normal proof of $ M(s) = v$ in ZFC is a formula sequence $\pi = \{ {\varphi}_0, {\varphi}_1, {\varphi}_2, \ldots, {\varphi}_r \}$ such that: $${\varphi}_r \text{ is the formula: } \exists i t_{i,1}=(v,q_{halt},1) \text{ and } \langle \textbf{T}_{\langle M,s \rangle}, \textbf{AL} \rangle {\vdash}^{\pi} {\varphi}_r$$
	The ASCII length of the normal proof be denoted by $\|\pi\|_{as}$, defined as $$\|\pi\|_{as} = \sum_{i=0}^{r}\textbf{asciilen}({\varphi}_{i})$$
\end{defn}

\begin{defn}
	\label{string_order}(\textbf{string order})
	Let $x$ and $y$ are two ASCII strings, we say that $x$ precede $y$, written $x <_s y$, if $\|x\|_{as} < \|y\|_{as}$, or  $\|x\|_{as} = \|y\|_{as}$ and $x$ precede $y$ in dictionary order.
\end{defn}

\begin{defn}
	\label{join_str}
	Let $x$ and $y$ are two ASCII strings, and $\|x\|_{as}=n$, $\|y\|_{as}=m$, the concatenation of $x$ and $y$, written $x \circ y$, is the string obtained by appending $y$ to the end of $x$, i.e., $x \circ y = x_1\cdots x_n y_1 \cdots y_m $.
\end{defn}

\begin{defn}
	\label{concatenate}
	Let $S_1$ and $S_2$ are two finite formula sequences, such that: $$S_1=\langle \varphi_{0},\varphi_{1},\varphi_{2},\ldots,\varphi_{n} \rangle$$ $$S_2=\langle \sigma_{0},\sigma_{1},\sigma_{2},\ldots,\sigma_{m} \rangle$$
	the concatenation of $S_1$ and $S_2$, written $S_1 + S_2$, is the finite formula sequence obtained by appending $S_2$ to the end of $S_1$: $$S_1 + S_2 = \langle \varphi_{0},\varphi_{1},\varphi_{2},\ldots,\varphi_{n},\sigma_{0},\sigma_{1},\sigma_{2},\ldots,\sigma_{m} \rangle$$
	In general, $S_1 + S_2 \ne S_2 + S_1$, and it is obvious that the operation ``+'' satisfy associative law. So if $S_0,S_1,S_2,\ldots,S_n$ are all finite formula sequences, we can define $$S_0+S_1+S_2+\ldots +S_n = (\ldots((S_0+S_1)+S_2)+\ldots) +S_n$$denoted by $$\sum_{i=0}^{n}S_i$$
\end{defn}

\begin{defn}
	\label{seq_order}(\textbf{sequence order})
	Let $S_1$ and $S_2$ are two finite formula sequences, such that: $$S_1=\langle \varphi_{0},\varphi_{1},\varphi_{2},\ldots,\varphi_{n} \rangle $$ $$S_2=\langle \sigma_{0},\sigma_{1},\sigma_{2},\ldots,\sigma_{m} \rangle $$
	and we take each formula as an ASCII string, say that $S_1$ precede $S_2$, written $S_1 <_s S_2$, if and only if $$ \varphi_{0}\circ \varphi_{1}\circ \varphi_{2} \ldots,\circ \varphi_{n} <_s \sigma_{0} \circ \sigma_{1} \circ \sigma_{2} \ldots \circ \sigma_{m}$$
\end{defn}

%Let  $O_{\langle M, w \rangle}$ as definition \ref{oms}. It is obvious that if $ O_{\langle M, w \rangle}(s) = 1$, then there is a proof of $ O_{\langle M, w \rangle}(s) = 1$ in ZFC and if $ O_{\langle M, w \rangle}(s) = 0$, there is a proof of $ O_{\langle M, w \rangle}(s) = 0$.

%Let $M$ is a Turing machine,  consider a special form of proof sequence of $ M(s) = v$ where $v \in \{0,1\}$: $sq = \{ {\varphi}_0, {\varphi}_1, {\varphi}_2, \ldots, {\varphi}_r \}$. $sq$ consists of three parts formula sequences as the following definition:

\begin{remark}
	\label{fsd}
	Let $\pi_1$ and $\pi_2$ are two normal form proofs of $M(s)=v$, if $\| \pi_1 \|_{as} < \| \pi_2 \|_{as}$, we say $\pi_1$ is shorter than $\pi_2$. Indeed, if $M(s)=v$, the process of $M$ computing on $s$ can be easily converted to a normal proof $\pi$ such that $ \langle \textbf{T}_{\langle M,s \rangle}, \textbf{AL} \rangle {\vdash}^{\pi} \exists i\ (t_{i,1}=(v,q_{halt},1) )$. Let its ASCII length is $\| \pi \|_{as} = n$. Since the normal proofs of $M(s)=v$ with ASCII length shorter than $n$ are finite, we can enumerate formula sequences in sequence order, and check whether it is a normal proof of $M(s)=v$. Hence it is easy to see that there is an algorithm, for any $M(s)=v$, it give the minimum of the set $S = \{k|\  k=\| \pi \|_{as},\ \pi \text{ is a proof of } M(s) = v \text{ in ZFC }\}$.
\end{remark}

\begin{defn}
	\label{fs} (\textbf{FS})\ 
	Fix an algorithm which can find the shortest ASCII length of normal proofs as in the \textbf{remark} \ref{fsd} , throughout this article denote it by $\textbf{FS}(M,s,v)$, abbreviated  $\textbf{FS}(M,s)$, if the machine $M$ halts on all inputs, i.e., for any $M(s)=v$, $$\textbf{FS}(M,s,v) = \textbf{min} \{k|\  k=\| \pi \|_{as},\ \langle \textbf{T}_{\langle M,s \rangle}, \textbf{AL} \rangle {\vdash}^{\pi} \exists i\ (t_{i,1}=(v,q_{halt},1) ) \} $$
	
	$\textbf{FS}(M,s)$= ``On input $\langle M,s \rangle$, an encoding of a machine $M$ and a string $s$:
	\begin{enumerate}
		\item Using the description of $M$ and $s$, compute $M$ on $s$ and get result $v$.
		\item Enumerate formula sequence in the sequence order, every time that $\textbf{FS}(M,s)$ outputs a sequence $S$, verify whether it is satisfies $$\langle \textbf{T}_{\langle M,s \rangle}, \textbf{AL} \rangle {\vdash}^{S} \exists i\ (t_{i,1}=(v,q_{halt},1) )$$
		\item Let the first formula sequence $\pi$ satisfies $$\langle \textbf{T}_{\langle M,s \rangle}, \textbf{AL} \rangle {\vdash}^{\pi} \exists i\ (t_{i,1}=(v,q_{halt},1) )$$
		then return the result $\|\pi\|_{as}$''
	\end{enumerate}
	
\end{defn}

\begin{defn}
	\label{adjoint_proof_complexity} (\textbf{adjoin proof complexity})\
	Let $M$ be a Turing machine that halts on all inputs. The adjoint proof complexity of $M$ is the function $f: \omega \rightarrow \omega$, where $f(n)$ is the maximum number of the set: $\{\textbf{FS}(M,r)|\ \|r\|=n \}$, denote such function by $\textbf{apf}_M$, i.e., if $\|s\|=n$ then $$\textbf{apf}_M(\|s\|)=\textbf{apf}_M(n)=f(n) = \textbf{max}\{ \textbf{FS}(M,r)|\ \|r\|=\|s\| \} $$
\end{defn}

Let $M$ be a Turing machine which compute a function $$g: \{0,1\}^* \rightarrow \{0,1\}$$ and $M(s)=v,\ \|s\|=n$. $t$ is the table of $M$ computing on input $s$ and the time complexity of $M$ is $f(n)$, then the process of $M$ computing on input $s$ can be converted to a special normal proof of $ M(s) = v$ as the following.

%First, there is one and only one tape square be reading at any time by the machine as in the \textbf{definition} \ref{def_tm_and_input}: $$\forall i \exists k   (t_{i,k} \in \Gamma \times Q \times \{1\})  \wedge (\forall j  (j \ne k) \rightarrow (t_{i,j} \in \Gamma \times Q \times \{0\}) )$$ Thus, for each row $t_i$ of the table $t$, there is only one tape square in $\Gamma \times Q \times \{1\}$, denoted by $t_{i,k_i}$, i.e., $t_{i,k_i} \in \Gamma \times Q \times \{1\}$, and from an appropriate transition rule of the machine $M$, we can decide the next tape square $t_{i+1,k_{i+1}}$ in row $t_{i+1}$ which satisfys $t_{i+1,k_{i+1}} \in \Gamma \times Q \times \{1\}$.

\begin{defn}
	\label{special_proof}
	Let $M$ be a Turing machine which compute a function $$g: \{0,1\}^* \rightarrow \{0,1\}$$ and $M(s)=v,\ \|s\|=n$, ``$t$'' is the table of $M$ computing on input $s$, the time complexity of $M$ is $f(n)$. It is obvious that there are only finite tape squares be affected by the computation, exactly not exceed $(f(n)+1) \times (f(n)+1)$ tape squares. %Let $$Tb_{s} = \{ t(a)|\ \ 0 \le a \le f(n) \}$$ and 
	
	The content of each tape square is determined by certain squares in the preceding row. If we know the values at $t_{(i-1),(j-1)}, t_{(i-1),(j)},\text{ and } t_{(i-1),(j+1)} $, we
	can obtain the value at $t_{i,j}$ with $M$'s transition function. For example:
	
	Let $100< f(n)+1$, and if we have proved the formula: $$t_{99,100} = (0,q,1)$$and a transition rule is:
	
	\begin{eqnarray}
	&\forall i \forall j (t_{i,j} = (0, q, 1) \rightarrow \nonumber \\
	&( t_{(i+1),j} = (1,p,0) )\nonumber \\
	&\wedge( t_{(i+1),(j+1)} = (\pi_{M}(t_{i,(j+1)}),p,1) )\nonumber \\
	&\wedge (\forall k (k \ne j \wedge k \ne j+1) \rightarrow t_{(i+1),k} = (\pi_{M}(t_{i,k}),p,0)) )\label{ex_02}
	\end{eqnarray}
	Then we can prove the formula: $t_{100,100} = (1,p,0)$ from the above two formulas, let $\gamma$ denotes the formula (\ref{ex_02}), $\pi$ denotes a formula sequence, a special normal proof of $ M(s) = v$,  the section of proving  $t_{100,100} = (1,p,0)$ as following:
	
	\begin{description}
		\item[$\pi(n_1)$:] $t_{99,100} = (0,q,1)$, \\(previously proved)
		\item \ldots
		\item \ldots
		\item \ldots
		\item[$\pi(c):$] $\gamma$, \\ ($ \textbf{G}_{\text{IN}}(\langle \textbf{T}_{\langle M,s \rangle}, \textbf{AL} \rangle, \pi(c))=1$, i.e., $\pi(c) \in \textbf{T}_{\langle M,s \rangle}$, indeed $\pi(c) \in \textbf{def}_M$ )
		\item[$\pi(c+1)$:] $\gamma \rightarrow ( \forall i \forall j t_{i,j} = (0, q, 1) \rightarrow  ( t_{(i+1),j} = (1,p,0) ) )$, \\ ($\textbf{G}_{\Lambda}({\pi}(c+1)) = 1$, i.e., ${\pi}(c+1)\in \Lambda$, indeed $\pi(c+1)$ is a tautology)
		\item[$\pi(c+2)$:] $\forall i \forall j t_{i,j} = (0, q, 1) \rightarrow  ( t_{(i+1),j} = (1,p,0) )$,  \\($\textbf{G}_{\rightarrow }(\pi,c+2,c+1,c) = 1$, i.e., $\pi(c+2)$ is obtained by modus ponens from $\pi(c+1)$ and $\pi(c)$)
		\item[$\pi(c+3)$:] $(\forall i \forall j t_{i,j} = (0, q, 1) \rightarrow  ( t_{(i+1),j} = (1,p,0) )) \rightarrow \\(t_{99,100} = (0, q, 1) \rightarrow  ( t_{100,100} = (1,p,0) ))$, \\ ($\textbf{G}_{\Lambda}({\pi}(c+3)) = 1$, i.e., ${\pi}(c+3)\in \Lambda $ )
		\item[$\pi(c+4)$:] $t_{99,100} = (0, q, 1) \rightarrow  ( t_{100,100} = (1,p,0) )$, \\($\textbf{G}_{\rightarrow }(\pi,c+4,c+3,c+2) = 1$, i.e., $\pi(c+4)$ is obtained by modus ponens from $\pi(c+3)$ and $\pi(c+2)$)
		\item[$\pi(c+5)$:] $t_{100,100} = (1,p,0)$ \\ ($\textbf{G}_{\rightarrow }(\pi,c+5,c+4,n_1) = 1$, i.e., $\pi(c+5)$ is obtained by modus ponens from $\pi(c+4)$ and $\pi(n_1)$)
	\end{description}
	The six formulas from $\pi(c)$ to $\pi(c+5)$ form a proof section of $t_{100,100} = (1,p,0)$, denoted by $\textbf{SEC}_{\pi}(t_{100,100} = (1,p,0))$.
	
	$t_{100,100}$ is represented by \$t\_\{100,100\}\$ and $t_{99,100}$ is represented by \$t\_\{99,100\}\$ in \textbf{\LaTeXe}, and it is obvious that $\| 99 \|_{as} < \| 100 \|_{as} < 100 < f(\|s\|)+1=f(n) +1 $\\therefore: $$\| t_{99,100} \|_{as} = \| \text{\$t\_\{99,100\}\$}  \|_{as} < 2f(n) +9$$ $$\| t_{100,100} \|_{as}  = \| \text{\$t\_\{100,100\}\$}  \|_{as} < 2f(n) +9$$
	Because $\textbf{def}_M$ is a sequence formulas $\langle{\varphi}_0,{\varphi}_1,{\varphi}_2,\ldots,{\varphi}_k\rangle$ as in the \textbf{definition} \ref{def_tm_and_input}, we can define $$\|\textbf{def}_M\|_{as} = \sum_{i=0}^{k}\|{\varphi}_i\|_{as}$$
	and it is easy to see
	
	\begin{enumerate}
		\item $\pi(c) \in \textbf{def}_M$, so $\| \pi(c) \|_{as} \le \| \textbf{def}_M \|_{as}$
		\item $\| \pi(c+1) \|_{as} < 2\| \pi(c) \|_{as} \le 2\| \textbf{def}_M \|_{as}$
		\item $\| \pi(c+2) \|_{as} < \| \pi(c) \|_{as} \le \| \textbf{def}_M \|_{as}$
		\item $\| \pi(c+3) \|_{as} < 2\| \pi(c+2) \|_{as} +  4f(n) + 20 < 2\| \textbf{def}_M \|_{as} + 4f(n) + 20 $
		\item $\| \pi(c+4) \|_{as} < \| \pi(c+3) \|_{as} < 2\| \textbf{def}_M \|_{as} + 4f(n) + 20$
		\item $\| \pi(c+5) \|_{as} < \| \pi(c+4) \|_{as} < 2\| \textbf{def}_M \|_{as} + 4f(n) + 20$
	\end{enumerate}

	$$ \sum_{i=0}^{5}\|\pi(c+i)\|_{as}  < 12f(n) + 10\| \textbf{def}_M \|_{as} + 60$$ %where $C$ is constant number large enough.
	Using the same approach as proving  $t_{100,100} = (1,p,0)$ above, we can prove a formula $t_{a,b} = v_{ab}$ for each pair $\langle a,b \rangle ,\ 0 \le a,b \le f(n)+1$, denoted by $\textbf{SEC}_{\pi}(t_{a,b} = v_{ab})$, is called the proof section of $t_{a,b} = v_{ab}$ where the $v_{ab}$ is the value of the tape square $t_{a,b}$ on the table of $M(s)=v$.
	
	The idea behind this approach is simple, the proof formula sequence is just a  description of $M$ computing on input $s$: the tape configuration  $t_i$ determined by the preceding tape configuration  $t_{i-1}$ and an appropriate transition rule of the machine $M$. Note that

	\begin{enumerate}
		\item $\textbf{SEC}_{\pi}(  t_{0,0} = (\triangleright, q_0, 1))$ is just the formula itself, because from the \textbf{definition} \ref{def_tm_and_input} $$ (t_{0,0} = (\triangleright, q_0, 1)) \in \textbf{T}_{\langle M,s \rangle}$$
		\item The same reasoning applies to any proof section of $$t_{0,b} = (s_b, q_0, 0),\ \ 0<b \le n+1 $$ i.e., $\textbf{SEC}_{\pi}(  t_{0,b} = (s_b, q_0, 0) )$ is just one formula, itself.
		\item From the \textbf{definition} \ref{def_tm_and_input}, any $b>n+1 $, $\textbf{SEC}_{\pi}(  t_{0,b} =  (\sqcup, q_0, 0) )$ is following formula sequence:
		$$b>n+1$$
		$$\forall j  (j > n+1) \rightarrow t_{0j} = (\sqcup, q_0, 0)$$		
		$$( \forall j  (j > n+1) \rightarrow t_{0j} = (\sqcup, q_0, 0) ) \rightarrow ( (b > n+1) \rightarrow t_{0,b} = (\sqcup, q_0, 0) ) $$
		$$(b > n+1) \rightarrow t_{0,b} = (\sqcup, q_0, 0)$$
		$$t_{0,b} =  (\sqcup, q_0, 0)$$
		\item For any $a>0$, the $\textbf{SEC}_{\pi}(t_{a,b} = v_{ab})$ like the case $\textbf{SEC}_{\pi}(t_{100,100} = (1,p,0))$ shown above, is a description of how the content of the tape square $t_{a,b}$ be determined by certain squares in the preceding row.
	\end{enumerate}
	Hence it is easy to see that there exist two numbers $K$ and $C$, independent of the input $s$, for all $$0 \le a,b \le f(\|s\|)+1=f(n)+1,\  \| \textbf{SEC}_{\pi}(t_{a,b} = v_{ab}) \|_{as} < Kf(n) + C  $$

	Since $t$ is the table of $M(s)=v$, there exists a number $d \le f(n)+1$ satisfys  $t_{d,1}=(v,q_{halt},1)$, and we can prove the formula  $t_{d,1}=(v,q_{halt},1)$ like in the described situation  $t_{99,100} = (0,q,1)$ above. Then $$(t_{d,1}=(v,q_{halt},1)) \rightarrow \exists i t_{i,1}=(v,q_{halt},1),\ \ \ (\text{denoted by }{\varphi}_{r-1})$$ $$\exists i t_{i,1}=(v,q_{halt},1),\ \ \  (\text{denoted by }{\varphi}_r)$$ is the proof of $\exists i t_{i,1}=(v,q_{halt},1)$ from $t_{d,1}=(v,q_{halt},1)$. Obviously, $$ \|\varphi_{r-1}\|_{as} + \|\varphi_r\|_{as} < Kf(n) + C $$

	Thus there is a special normal proof of $ M(s) = v$ in ZFC, such that: 
	
	\begin{enumerate}
		\item $$\pi = \{ {\varphi}_0, {\varphi}_1, {\varphi}_2, \ldots, {\varphi}_r \}\\ \text{ \ }= (\sum_{i=0}^{f(n)+1} \sum_{j=0}^{f(n)+1}\textbf{SEC}_{\pi}(t_{i,j} = v_{ij})\ ) + \langle {\varphi}_{r-1} \rangle + \langle {\varphi}_r \rangle $$
		
		Note that the operation ``+'' and ``$\sum$'' on formula sequences are defined in \textbf{definition} \ref{concatenate}.% $\textbf{SEC}_{\pi}(t_{i,j} = v_{ij})$ like the $\textbf{SEC}_{\pi}(t_{100,100} = (1,p,0))$ described above, a proof section of $t_{i,j} = v_{ij}$ from a tape square in the preceding row and a transition rule of the machine.

		\item ${\varphi}_{r-1} \text{ is the formula: } (t_{d,1}=(v,q_{halt},1)) \rightarrow \exists i t_{i,1}=(v,q_{halt},1) $.
		\item ${\varphi}_r \text{ is the formula: } \exists i t_{i,1}=(v,q_{halt},1) $.
		
		\item $\langle \textbf{T}_{\langle M,s \rangle}, \textbf{AL} \rangle {\vdash}^{\pi} {\varphi}_r$
	\end{enumerate}
	We denote this special normal proof of $ M(s) = v$ as $\Pi_{\langle M,s \rangle}$. Therefore, $$ \| \Pi_{\langle M,s \rangle} \|_{as} = \sum_{i=0}^{f(\|s\|)+1} \sum_{j=0}^{f(\|s\|)+1}\|\textbf{SEC}_{\pi}(t_{i,j} = v_{ij})\|_{as} + \| {\varphi}_{r-1}  \|_{as} + \|  {\varphi}_r  \|_{as} $$\\$$<  [\sum_{i=0}^{f(\|s\|)+1} \sum_{j=0}^{f(\|s\|)+1} (Kf(\|s\|) + C)] + Kf(\|s\|) + C$$\\$$ = [(f(\|s\|)+2)(f(\|s\|)+2)+1](Kf(\|s\|)+C)$$Where the two numbers $K$ and $C$ are independent of the input $s$.
	%We denote this special normal proof of $ M(s) = v$ as $\Pi_{\langle M,s \rangle}$. Therefore, $$ \| \Pi_{\langle M,s \rangle} \|_{as} = \sum_{i=0}^{f(n)+1} \sum_{j=0}^{f(n)+1}\|\textbf{SEC}_{\pi}(t_{i,j} = v_{ij})\|_{as} + \| {\varphi}_{r-1}  \|_{as} + \|  {\varphi}_r  \|_{as} $$\\$$<  [\sum_{i=0}^{f(n)+1} \sum_{j=0}^{f(n)+1} (Kf(n) + C)] + Kf(n) + C$$\\$$ = [(f(n)+2)(f(n)+2)+1](Kf(n)+C)$$Where the two numbers $K$ and $C$ are independent of the input $s$.

\end{defn}
It is obvious that 
\begin{equation}
	\label{fs_ineq}
	\textbf{FS}(M,s) \le \| \Pi_{\langle M,s \rangle} \|_{as} < [(f(\|s\|)+2)(f(\|s\|)+2)+1](Kf(\|s\|)+C)
\end{equation}
Therefore we get the following lemma:

\begin{lemma}
	\label{poly_proof_comlexity} (\textbf{polynomial proof complexity})\\
	Let $M$ be a polynomial time Turing machine. then its adjoint proof complexity is also a polynomial, i.e., $\textbf{apf}_M$ is bounded by a polynomial.
\end{lemma}

\begin{proof}
	Let the time complexity of $M$ is a polynomial $f(n)$.  From the \textbf{definition} \ref{adjoint_proof_complexity}, $\textbf{apf}_M(\|s\|) = \textbf{max}\{ \textbf{FS}(M,r)|\ \|r\|=\|s\| \} $. since the above inequality (\ref{fs_ineq}), we get $$\textbf{apf}_M(\|s\|) <  [(f(\|s\|)+2)(f(\|s\|)+2)+1](Kf(\|s\|)+C)$$
\end{proof}

\begin{lemma}
	\label{bounded_running_time}
	(\textbf{bounded running time})\\Let $ M $ is a Turing machine $(Q,\Gamma,\delta,q_0,q_{halt})$, $v \in \{0,1\}, r \in \{0,1\}^*$, and \textbf{AL} is the algorithm as the \textbf{definition} \ref{searching_al}, $ M(r) = v$, $\pi$ is a formula sequence:$$  \langle {\varphi}_0, {\varphi}_1, {\varphi}_2, \ldots, {\varphi}_n \rangle $$ such that: $${\varphi}_n \text{ is the formula: } \exists i t_{i,1}=(v,q_{halt},1) \text{ and } \langle \textbf{T}_{\langle M,r \rangle}, \textbf{AL} \rangle {\vdash}^{\pi} {\varphi}_n$$
	That is $\pi$ is a normal proof of $ M(r) = v$, thus we can define a Turing machine on $s \in \{0,1\}^*$ as: $$\textbf{f}(s) = \textbf{CK}_{\pi}(\langle \textbf{T}_{\langle M,s \rangle}, \textbf{AL} \rangle, \pi)$$
	then the time complexity of $\textbf{f}$: $t_{\textbf{f}}(n)$ is bounded, i.e., there exists a number $K$, for all $s \in \{0,1\}^*$, $t_{\textbf{f}}(\|s\|) < K$.
\end{lemma}

\begin{proof}
	From the \textbf{definition} \ref{adjoint_checker} $\textbf{CK}_{\pi}$, the adjoint checker of $\pi$, be described by a group of checkers:  $$\{g_0,g_1,g_2,\ldots,g_n\}$$
	therefore $f(s)=1$ if and only if the formula sequence $\pi$ is a normal proof of $M(s)=1$.
	Indeed there are only five types of checkers:
	
	\begin{enumerate}
		\item $\textbf{G}_{\Lambda}$ type;
		\item $\textbf{G}_{\text{ZFC}}$ type;
		\item $\textbf{G}_{\text{IN}}$ type;
		\item $\textbf{G}_{\rightarrow }$ type;
		\item $\textbf{G}_{\forall}$ type.
	\end{enumerate}
	Since the formula sequence $\pi$ is fixed $$ \pi = \langle {\varphi}_0, {\varphi}_1, {\varphi}_2, \ldots, {\varphi}_n \rangle $$ and $\langle \textbf{T}_{\langle M,r \rangle}, \textbf{AL} \rangle {\vdash}^{\pi} {\varphi}_n$, therefore from the \textbf{theorem} \ref{self_check}: $$f(r) = \textbf{CK}_{\pi}(\langle \textbf{T}_{\langle M,r \rangle}, \textbf{AL} \rangle, \pi) =1$$So only the $\textbf{G}_{\text{IN}}$ type checkers need to be computed, because:
	
	\begin{enumerate}
		\item if $g_i = \textbf{G}_{\Lambda}(sq(i))$, from the \textbf{definition} \ref{proof type}, $\pi(i)$, i.e., ${\varphi}_i$ must be in $\Lambda$, therefore $\textbf{G}_{\Lambda}(\pi(i)) = 1$,  $\textbf{CK}_{\pi}$ need not to compute the checker   $g_i = \textbf{G}_{\Lambda}(sq(i))$ on input $(\langle \textbf{T}_{\langle M,s \rangle}, \textbf{AL} \rangle, \pi)$;
		
		\item if $g_i = \textbf{G}_{\text{ZFC}}(sq(i))$,  from the \textbf{definition} \ref{proof type}, $\pi(i)$, i.e., ${\varphi}_i$ must be in ZFC, therefore $\textbf{G}_{\text{ZFC}}(\pi(i)) = 1$,  $\textbf{CK}_{\pi}$ need not to compute the checker   $g_i = \textbf{G}_{\text{ZFC}}(sq(i))$ on input $(\langle \textbf{T}_{\langle M,s \rangle}, \textbf{AL} \rangle, \pi)$;

		\item if $g_i = \textbf{G}_{\rightarrow }(sq,i,j,k)$, from the \textbf{definition} \ref{proof type}, $\pi(i)$, i.e., ${\varphi}_i$ must satisfys the following condition $ j,k < i,\ {\varphi}_j = {\varphi}_k \rightarrow {\varphi}_i$ therefore $\textbf{G}_{\rightarrow }(\pi,i,j,k)=1$, $\textbf{CK}_{\pi}$ need not to compute the checker $g_i = \textbf{G}_{\rightarrow }(sq,i,j,k)$ on input $(\langle \textbf{T}_{\langle M,s \rangle}, \textbf{AL} \rangle, \pi)$;
		
		\item if $g_i = \textbf{G}_{\forall}(sq,i,j,k)$,  from the \textbf{definition} \ref{proof type}, $\pi(i)$, i.e., ${\varphi}_i$ must satisfys the following condition $j < i,\  k \in \omega,\  ({\varphi}_i = \forall x_k {\varphi}_j) $, therefore $\textbf{G}_{\forall}(\pi,i,j,k) = 1$, $\textbf{CK}_{\pi}$ need not to compute the checker $g_i = \textbf{G}_{\forall}(sq,i,j,k)$ on input $(\langle \textbf{T}_{\langle M,s \rangle}, \textbf{AL} \rangle, \pi)$;
	\end{enumerate}
	Thus the value of  $\textbf{CK}_{\pi}(\langle \textbf{T}_{\langle M,s \rangle}, \textbf{AL} \rangle, \pi)$, i.e., $\textbf{f}(s)$ depends only on $\textbf{G}_{\text{IN}}$ type checkers. 
	
	Let  $g_i = \textbf{G}_{\text{IN}}(S, sq(i))$, then from the \textbf{definition} \ref{adjoint_checker}, $\textbf{CK}_{\pi}$ compute $$ \textbf{G}_{\text{IN}}(\langle \textbf{T}_{\langle M,s \rangle}, \textbf{AL} \rangle, \pi(i))$$ and from the \textbf{definition} \ref{compute_checker}, $\textbf{G}_{\text{IN}}$ use the algorithm $\textbf{AL}$ to decide whether ${\pi}(i) \in \textbf{T}_{\langle M,s \rangle}$. From the \textbf{definition} \ref{def_tm_and_input}, $$\textbf{T}_{\langle M,r \rangle} = \langle \textbf{def}_M \cup \textbf{input}_r \rangle $$  $$\textbf{T}_{\langle M,s \rangle} = \langle \textbf{def}_M \cup \textbf{input}_s \rangle $$
	Let $ \textbf{def}_M $ are $k$ formulas, and $r = r_0r_1r_2 \ldots r_l,\ r_i \in  \{0,1\}, 0 \le i \le l$, therefore $\textbf{T}_{\langle M,r \rangle}$ are $k+l+4$ formulas as in  \textbf{definition} \ref{def_tm_and_input}.
	
	Since $\langle \textbf{T}_{\langle M,r \rangle}, \textbf{AL} \rangle {\vdash}^{\pi} {\varphi}_n$ and $g_i$ is a $\textbf{G}_{\text{IN}}$ type checker, from the \textbf{definition} \ref{proof type}, $\pi(i) \in \textbf{T}_{\langle M,r \rangle}$.
	
	Therefore, for each $s \in \{0,1\}^*$, when $\textbf{G}_{\text{IN}}$ use the algorithm $\textbf{AL}$ to decide whether ${\pi}(i) \in \textbf{T}_{\langle M,s \rangle}$, only the first $k+l+4$ formulas of the $\textbf{T}_{\langle M,s \rangle}$ need to be tested. That is the number of steps in compute a $\textbf{G}_{\text{IN}}$ type checker is less than a fixed number, denoted by $C$, and the number of the all $\textbf{G}_{\text{IN}}$ type checkers less than $n+1$, compute all the all $\textbf{G}_{\text{IN}}$ type checkers less than $ C\times (n+1)$ steps.
	
	Hence there exists a number $K$, for all $s \in \{0,1\}^*$, the number of steps of compute $\textbf{f}(s) = \textbf{CK}_{\pi}(\langle \textbf{T}_{\langle M,s \rangle}, \textbf{AL} \rangle, \pi)$ is less than $K$, i.e., $\forall s\ s \in \{0,1\}^*\ t_{\textbf{f}}(\|s\|) < K$.
\end{proof}

In order to analyze proof procedure in more detail, we now consider the input of normal proof. Let $M$ be a Turing machine on $\{0,1\}^*$, $r = r_0r_1r_2\ldots r_l,\ r_i \in \{0,1\}, 0 \leq i \leq l$, $M(r)=1$,  and $\pi$ is a normal proof of $M(r)=1$, from the \textbf{definition} \ref{normal_proof}, we know $$\langle \textbf{T}_{\langle M,r \rangle}, \textbf{AL} \rangle {\vdash}^{\pi} \exists i t_{i,1}=(1,q_{halt},1)$$According to the \textbf{definition} \ref{def_tm_and_input} the $\langle \textbf{T}_{\langle M,r \rangle}$ is:

\begin{itemize}
	\item ${\varphi}_{0}$,
	\item ${\varphi}_{1}$,
	\item ${\varphi}_{2}$,
	\item \ldots,
	\item ${\varphi}_{k}   \text{ is } \exists m (\forall j  (j < m) \rightarrow t_{0,j} \ne (\sqcup, q_0, 0) ) \wedge (\forall j  (j \ge m) \rightarrow t_{0,j} = (\sqcup, q_0, 0) ) $,
	\item ${\varphi}_{k+1} \text{ is } \forall j  (j > l+1) \rightarrow t_{0,j} = (\sqcup, q_0, 0)$,
	\item ${\varphi}_{k+2} \text{ is } t_{0,0} = (\triangleright, q_0, 1)$,
	\item ${\varphi}_{k+3} \text{ is } t_{0,1} = (r_0, q_0, 0)$,
	\item ${\varphi}_{k+4} \text{ is } t_{0,2} = (r_1, q_0, 0)$,
	\item ${\varphi}_{k+5} \text{ is } t_{0,3} = (r_2, q_0, 0)$,
	\item \ldots,
	\item ${\varphi}_{k+l+3} \text{ is } t_{0,(l+1)} = (r_l, q_0, 0)$.
	
\end{itemize}
It is easy to see that for any $s \in \{0,1\}^*$, the first $k+1$ formulas of $ \textbf{T}_{\langle M,s \rangle}$ are the same formulas, i.e., the formula sequence $\textbf{def}_M$. If $s \ne r$ the different formulas between $ \textbf{T}_{\langle M,s \rangle}$ and $ \textbf{T}_{\langle M,r \rangle}$ are all in $$\textbf{input}_s \cup \textbf{input}_r $$Therefore we have the following definition:

\begin{defn}
	Let  $\pi = \langle {\varphi}_0, {\varphi}_1, {\varphi}_2, \ldots, {\varphi}_n \rangle$ is a normal proof of $M(r)=1$ as described above, the key information set of $\pi$ is the formula set, denoted by $\textbf{keyset}(\pi)$:$$\{\ \varphi |\  (\varphi \in \pi) \wedge (\varphi \in \textbf{input}_r)  \}$$and the key information of $\pi$, is the formula obtained by connecting all the formulas of $\textbf{keyset}(\pi)$ by $\wedge$ operations, and denoted by $\textbf{keyinfo}(\pi)$:$$\bigwedge_{\varphi \in \textbf{keyset}(\pi) } \varphi $$ 
\end{defn}

\begin{corollary}
	\label{ckeqkeyinfo}
	Let $\pi = \langle {\varphi}_0, {\varphi}_1, {\varphi}_2, \ldots, {\varphi}_n \rangle$ is a normal proof of $M(r)=1$, then $\textbf{CK}_{\pi}(\langle \textbf{T}_{\langle M,s \rangle}, \textbf{AL} \rangle, \pi) =1 $ if and only if the input $s \in \{0,1\}^*$ satisfies $\textbf{keyinfo}(\pi)$, i.e., $$\forall s ( (\textbf{CK}_{\pi}(\langle \textbf{T}_{\langle M,s \rangle}, \textbf{AL} \rangle, \pi) =1) \leftrightarrow \textbf{keyinfo}(\pi) ) $$
	
\end{corollary}

\begin{proof}
	From the proof of \textbf{lemma} \ref{bounded_running_time}, we know that in order to decide whether or not $\textbf{CK}_{\pi}(\langle \textbf{T}_{\langle M,s \rangle}, \textbf{AL} \rangle, \pi) =1 $, only the $\textbf{G}_{\text{IN}}$ type checkers need to be computed, and from the above discussion, indeed, only the checkers corresponding to the formulas in $\textbf{keyset}(\pi)$ need to be computed, therefore the lemma is proved.
\end{proof}

\begin{corollary}
	\label{keyinfotoms1}
	Let $\pi = \langle {\varphi}_0, {\varphi}_1, {\varphi}_2, \ldots, {\varphi}_n \rangle$ is a normal proof of $M(r)=1$, if an input $s \in \{0,1\}^*$ satisfies $\textbf{keyinfo}(\pi)$ then $M(s)=1$ can be proved in ZFC, i.e., the following formula can be proved in ZFC: $$\forall s (\textbf{keyinfo}(\pi) \rightarrow (M(s)=1) )$$
	
\end{corollary}

\begin{proof}
	It is obviously true from the above discussion.
\end{proof}

\begin{theorem}
	\label{ck1toms1}
	Let $\pi = \langle {\varphi}_0, {\varphi}_1, {\varphi}_2, \ldots, {\varphi}_n \rangle$ is a normal proof of $M(r)=1$, then the following formula can be proved in ZFC:
	$$\forall s ( (\textbf{CK}_{\pi}(\langle \textbf{T}_{\langle M,s \rangle}, \textbf{AL} \rangle, \pi) =1 ) \rightarrow ( M(s)=1) )$$
\end{theorem}

\begin{proof}
	This theorem is obviously deduced from the \textbf{corollary} \ref{ckeqkeyinfo} and \textbf{corollary} \ref{keyinfotoms1}.
\end{proof}

\begin{defn}
	\label{CK_C}
	Let $ M $ is a Turing machine $(Q,\Gamma,\delta,q_0,q_{halt})$, $v_i \in \{0,1\}, r_i \in \{0,1\}^*, i = 0,1,\cdots, n. $, and \textbf{AL} is the algorithm as the \textbf{definition} \ref{searching_al}, $ M(r_i) = v_i$, ${\pi}_i$ is a formula sequence:$$  \langle {\varphi}_{i_0}, {\varphi}_{i_1}, {\varphi}_{i_2}, \ldots, {\varphi}_{i_{k_i}} \rangle $$ such that: $${\varphi}_{i_{k_i}} \text{ is the formula: } \exists j t_{j,1}=(v_i,q_{halt},1) \text{ and } \langle \textbf{T}_{\langle M,r_i \rangle}, \textbf{AL} \rangle {\vdash}^{{\pi}_i} {\varphi}_{i_{k_i}}$$
	That is ${\pi}_i$ is a normal proof of $ M(r_i) = v_i$, the corresponding adjoint checker of ${\pi}_i$ is $\textbf{CK}_{{\pi}_i}$, and let $$C = \{\textbf{CK}_{{\pi}_0},\textbf{CK}_{{\pi}_2},\cdots,\textbf{CK}_{{\pi}_n} \}$$thus we can define a Turing machine $\textbf{F}_{C}$ on $s \in \{0,1\}^*$ as:\\
	$\textbf{F}_{C}$ = ``On input $s$, where $s \in \{0,1\}^*$:
	
	\begin{enumerate}
		\item for each $i,\ 0 \le i \le n $ use $\textbf{CK}_{\pi_i}$ to compute $$\textbf{CK}_{\pi_i}(\langle \textbf{T}_{\langle M,s \rangle}, \textbf{AL} \rangle, \pi_i)$$
		\item If there is a checker return 1, i.e., there is a normal proof $\pi_k$,  $\textbf{CK}_{\pi_k}(\langle \textbf{T}_{\langle M,s \rangle}, \textbf{AL} \rangle, \pi_k) = 1$, the machine $\textbf{F}_{C}$ return 1, and halts.
		\item If all the computations of checkes return 0, i.e., $$\textbf{CK}_{\pi_i}(\langle \textbf{T}_{\langle M,s \rangle}, \textbf{AL} \rangle, \pi_i) = 0, \text{ for all }0 \le i \le n $$ the machine $\textbf{F}_{C}$ return 0, and halts.''
	\end{enumerate}
	We call $\textbf{F}_{C}$ the $C$ generated verifier.
\end{defn}

\begin{corollary}
	\label{boundtimeofCKC}
	Let the Turing machine $\textbf{F}_{C}$ is the $C$ generated verifier as the \textbf{definition} \ref{CK_C}, then the time complexity of $\textbf{F}_{C}$: $t_{\textbf{F}_{C}}(n)$ is bounded, i.e., there exists a number $K$, for all $s \in \{0,1\}^*$, $t_{\textbf{F}_{C}}(\|s\|) < K$.
\end{corollary}

\begin{proof}
	From the \textbf{lemma} \ref{bounded_running_time}, we know that the time complexity of each  $$\textbf{CK}_{\pi_i}(\langle \textbf{T}_{\langle M,s \rangle}, \textbf{AL} \rangle, \pi_i)$$is bounded by a number $K_i$, hence the computation steps of $\textbf{F}_{C}$ is no more than $\sum_{i=0}^{n} K_i + M$ where $M$ is a large enough constant number.
\end{proof}

\begin{lemma}
	\label{fc1eqvckp}
	Let the Turing machine $\textbf{F}_{C}$ is the $C$ generated verifier as the \textbf{definition} \ref{CK_C}, then the following formula can be proved in ZFC:$$(\textbf{F}_{C}(s)=1) \leftrightarrow \biggl( \bigvee_{\textbf{CK}_{\pi_i} \in C} ( \textbf{CK}_{\pi_i}(\langle \textbf{T}_{\langle M,s \rangle}, \textbf{AL} \rangle, \pi_i)=1) \biggr) $$
\end{lemma}

\begin{proof}
	From the definition of $\textbf{F}_{C}$, it is obviously true.
\end{proof}

\begin{lemma}
	\label{reftocth}
	There exists a Turing machine $M$ that it halts on every input and the following five formulas can be proved in ZFC:
	
	\begin{enumerate}
		\item $\forall r,s \in \{0,1\}^*\   (\ \|s\| = \|r\| \rightarrow (  M(s) = M(r) )\ )$.
		\item $\forall r,s \in \{0,1\}^*\   (\ \|s\|> \|r\|) \rightarrow ( M(s)=1 \rightarrow  M(r)=1  )$.
		\item $\forall r,s \in \{0,1\}^*\   (\ \|s\|> \|r\|) \rightarrow ( M(r)=0 \rightarrow  M(s)=0  )$.
		\item $\forall s \in \{0,1\}^*\   (\  M(s)=0 \vee  M(s)=1  )$.
		\item $(\forall s (s \in \{0,1\}^*) \rightarrow ( M(s) = 1) ) \vee (\exists m(\|s\| < m \rightarrow M(s)=1 ) \wedge (\|s\| \ge m \rightarrow M(s)=0) )$.
	\end{enumerate}
	but the formula $ \forall s  ( M(s) = 1) $ is independent of ZFC, i.e., it cannot be proved in ZFC and its negation is also unprovable in ZFC.
\end{lemma} 

\begin{proof}
	This lemma is just the \textbf{Corollary 3.2} in the paper \cite{cth}.
\end{proof}

\begin{theorem}
	\label{keytheorem}
	Let $M$ as in the \textbf{lemma} \ref{reftocth}, i.e., $ \forall s  ( M(s) = 1) $ is independent of ZFC, then $\textbf{FS}(M,s,1)$ is unbounded on $s \in \{0,1\}^*$, that is
	$$\forall m \exists s \in \{0,1\}^*\ (\textbf{FS}(M,s,1) > m )$$
\end{theorem}

\textbf{PROOF IDEA}\ \ \   From the \textbf{lemma} \ref{reftocth}, $ \forall s  ( M(s) = 1) $ is independent of ZFC. Therefore we cannot find a string $  s \in \{0,1\}^* $ satisfying $ M(s) = 0  $, that is, for all $  s \in \{0,1\}^* $, we use $M$ to compute on $s$ will returning 1, but we cannot prove $ \forall s  ( M(s) = 1) $  in ZFC.

If $\textbf{FS}(M,s,1)$ is bounded on $s \in \{0,1\}^*$, then there exists a number $n$ for all $s \in \{0,1\}^*$,
$$\textbf{FS}(M,s,1) < n $$

Let  $$S_n = \{ \pi|\  \pi\ \text{is normal proof of }M(s) = 1,\ \ s \in \{0,1\}^*, \ \| \pi \|_{as}<n, \} $$
It is not hard to see that $S_n$ is finite.

Since $\textbf{FS}(M,s,1) < n $, for each $ s \in \{0,1\}^*$, there exists a normal proof sequence $\pi_s$ of $ M(s) = 1$ and $\|\pi_s\|_{as} < n$, $$ \langle \textbf{T}_{\langle M,s \rangle}, \textbf{AL} \rangle {\vdash}^{\pi_s} (\exists i t_{i,1}=(1,q_{halt},1))$$    therefore $\pi_s\in S_n$. Since $S_n$ is finite, thus we can prove $$  \forall s  ( M(s) = 1) $$ in finite steps in ZFC, contradiction.

\begin{proof}
	%$\textbf{T}_{\langle {O_{\langle M_h, w_h \rangle}}, s\rangle }$
		Let $M$ as in the \textbf{lemma} \ref{reftocth},if we found a string $s \in \{0,1\}^*,\  M(s) = 0$, then it is obvious that we can prove  $\neg( \forall s  ( M(s) = 1) ) $ in ZFC. But from the \textbf{lemma} \ref{reftocth}:
		
		\begin{equation}
		\label{e_001}
		\text{The statement }\forall s  ( M(s) = 1)\text{ is independent of ZFC.}
		\end{equation}
		
		Therefore we  cannot find such string, i.e., for all $  s \in \{0,1\}^* $, we use $M$ to compute on $s$ will returning 1, but we cannot prove $ \forall s  ( M(s) = 1) $  in ZFC.\\
		
		Now we prove the statement $\forall m \exists s \in \{0,1\}^*\ (\textbf{FS}(M,s,1) > m )$ by contradiction.\\
		
		First, we assume for the purpose of later obtaining a contradiction that $\textbf{FS}(M,s,1)$ is bounded on $s \in \{0,1\}^*$. Thus there exists a number $n$ for all $s \in \{0,1\}^*$,
		
		\begin{equation}
		\label{f_len}
			\textbf{FS}(M,s,1) < n.
		\end{equation}
		
		%$$\textbf{FS}(M,s,1) < n $$

		Let  $\phi$ is the formula $\exists i t_{i,1}=(1,q_{halt},1)$ then define the $S_n$ as:  $$S_n = \{ \pi|\ \langle \textbf{T}_{\langle M,s \rangle}, \textbf{AL} \rangle {\vdash}^{\pi} \phi,\ s \in \{0,1\}^*,\ \| \pi \|_{as}<n, \} $$
		
		It is not hard to see that $S_n$ is finite, therefore let $$S_n = \{ \pi_0, \pi_1, \pi_2,\ldots,\pi_m \}$$ and the corresponding adjoint checkers set $C_n = \{ c_{\pi}|\ c_{\pi} = \textbf{CK}_{\pi},\ \pi \in S_n \}$ is also finite, i.e., $$C_n = \{ \textbf{CK}_{\pi_0},\textbf{CK}_{\pi_1},\textbf{CK}_{\pi_2},\ldots,\textbf{CK}_{\pi_m} \}$$
		Let $\textbf{F}_{C_n}$ is the $C_n$ generated verifier(see the definition in \textbf{definition} \ref{CK_C}).
		
		From (\ref{f_len}), for each $ s \in \{0,1\}^*\ \textbf{FS}(M,s,1) < n $, therefore there exists a normal proof sequence $\pi_s$, $\|\pi_s\|_{as} < n$, $$ \langle \textbf{T}_{\langle M,s \rangle}, \textbf{AL} \rangle {\vdash}^{\pi_s} \phi$$
		Therefore $\pi_s \in S_n \text{ and the corresponding adjoint checker } \textbf{CK}_{\pi_s} \in C_n$ and from the \textbf{theorem} \ref{self_check}, we get $$\textbf{CK}_{\pi_s}(\langle \textbf{T}_{\langle M,s \rangle}, \textbf{AL} \rangle, \pi_s)=1.$$
		So when we compute the machine $\textbf{F}_{C_n}$ on any $ s \in \{0,1\}^*$, it will return 1.\\
		
		Because the \textbf{corollary} \ref{boundtimeofCKC}, the computation steps on $\textbf{F}_{C_n}$ is bounded, so there exists a number $K$, the computation complexity of $\textbf{F}_{C_n}$ is bounded by $K$: $$\forall s \in \{0,1\}^*\ t_{\textbf{F}_{C_n}}(\|s\|) < K$$and obviously, there are finite strings in $\{s|\ \|s\| \le K \}$, so we can prove the following formula  in ZFC: $$\|s\| \le K \rightarrow \textbf{F}_{C_n}(s) = 1$$ 
		
		Therefore we can prove the following two statements: $$\forall s \in \{0,1\}^*\  \|s\|\ge K \rightarrow t_{\textbf{F}_{C_n}}(\|s\|) < \|s\|$$and $$\|s\| \le K \rightarrow \textbf{F}_{C_n}(s) = 1$$Since the \textbf{theorem} \ref{short_time_tm}, we can prove $\forall s \textbf{F}_{C_n}(s)=1$ in ZFC, and from the \textbf{lemma} \ref{fc1eqvckp}, we can prove $$\forall s\biggl( \bigvee_{i=0}^{m} ( \textbf{CK}_{\pi_i}(\langle \textbf{T}_{\langle M,s \rangle}, \textbf{AL} \rangle, \pi_i)=1) \biggr) $$and since the \textbf{theorem} \ref{ck1toms1} we can prove $\forall s  ( M(s) = 1$ in ZFC, contradicting the statemenet of (\ref{e_001}): $\forall s  ( M(s) = 1)$ is independent of ZFC.

\end{proof}

Indeed, the \textbf{theorem} \ref{keytheorem} is actually the rigorous expression of ``there are essentially infinite different independent reasons govern the whole domain to serve the unprovable true statement''. From this theorem we get the following corollary:

\begin{corollary}
	The formula $ \forall s  ( M(s) = 1) $ is provable in ZFC, if and only if $\textbf{FS}(M,s,1)$ is bounded on $s \in \{0,1\}^*$, that is $$\exists m \forall s \in \{0,1\}^*\ (\textbf{FS}(M,s,1) < m )$$
\end{corollary}

\begin{proof}
	This statement is obviously true from the \textbf{theorem} \ref{keytheorem}.
\end{proof}

Indeed, the proof of \textbf{theorem} \ref{keytheorem} has shown a procedure how to search a proof of a general conclusion(for example $ \forall s \in \omega ( M(s) = 1) $), from some finite concrete examples. That is, the procedure is to find a adjoint checkers set $C_n $ which is large enough to satisfies the following two statements:  $$\forall s \in \{0,1\}^*\  \|s\|\ge K \rightarrow t_{\textbf{F}_{C_n}}(\|s\|) < \|s\|$$and $$\|s\| \le K \rightarrow \textbf{F}_{C_n}(s) = 1$$

Now we give this procedure as an explicit algorithm which will halts if and only if the formula $\forall s M(s)=1$ is not independant of ZFC:\\

\textbf{Algorithm}

\begin{enumerate}
	\item Begin with the  checkers set $C $ = empty set,
	\item \label{iteration} Since the \textbf{corollary} \ref{boundtimeofCKC}, we can find a number, $K$, such that $t_{\textbf{F}_{C}}(\|s\|) < K$  on any input $s$,
	\item Compute $M(s)$ and $\textbf{F}_C(s)$ on all $\|s\| \le K$,
	\item If  $\forall s (\|s\| \le K) \rightarrow M(s)=1 $ and $\|s\| \le K \rightarrow \textbf{F}_{C}(s) = 1$, then from the proof of \textbf{theorem} \ref{keytheorem}, $ \forall s  ( M(s) = 1) $ can be proved in ZFC, halts,
	\item Else if there exist $\|s\| \le K$ and $M(s)\ne 1$, then we can prove $$\neg \forall s ( M(s) = 1)$$ halts,
	\item \label{al_step6}Else if $\forall s (\|s\| \le K) \rightarrow M(s)=1 $, but $\exists r( \|r\| \le K) \wedge ( \textbf{F}_{C}(r) = 0) $, then adding an adjoint checker $\textbf{CK}_{\pi}$ to the  checkers set $C $, where  $\pi$ is a shortest normal proof of $M(r)=1$ and  $\textbf{CK}_{\pi}$ is the adjoint checker of $\pi$,
	\item goto \ref{iteration}

\end{enumerate}

In practice, we can improve this algorithm by using neural network technique at the step \ref{al_step6} to searching the shortest normal proof. The most interesting thing we will see in a later article is, that the algorithm seemingly to imply some sophisticated processes, such as training neural network,  cannot be proved being effective in formal system, though it is practically effective.


\begin{thebibliography}{10}
	
	
	\bibitem{bgs}
	T. P. Baker, J. Gill, and R. Solovay, \emph{Relativizations of the P=?NP question}, SIAM Journal on Computing 4(4):431-442, 1975.
	
	\bibitem{jh1}
	J. Hartmanis, \emph{Feasible Computations and Provable Complexity Problems}, SIAM, 1978.
	
	\bibitem{hnh}
	J. Hartmanis and J. Hopcroft, \emph{Independence results in computer science}, SIGACT News 8(4):13-24, 1976.
	
	\bibitem{jh2}
	J. Hartmanis, \emph{Independence results about context-free languages and lower bounds}, Information Proc. Lett. 20(5):241-248, 1985.
	
	\bibitem{hl}
	Harry R.Lewis and Christos H.Papadimitriou, \emph{Elements of the Theory of Computation 2nd Ed}, Prentice-Hall, 1998.
	
	\bibitem{ms}
	Michael Sipser, \emph{Introduction to the Theory of Computation 3rd Ed}, Cengage Learning, 2012.
	
	
	\bibitem{chk}
	C. C. Chang and H. J. Keislelr, \emph{Model Theory}, North-Holland, Amsterdam, 1990.
	
	\bibitem{hgs}
	W. Hodges, \emph{Model Theory}, Cambridge University Press, 1993.
	
	
	\bibitem{dm}
	David Marker, \emph{Model Theory: An Introduction}, Springer, 2002. 
	
	
	\bibitem{yim}
	Yu. I. Maninr, \emph{A Course in Mathematical Logic},(Graduate texts in mathematics; 53) Springer-Verlag, 1977.
	
	
	\bibitem{cth}
	Tianheng. Tsui, \emph{Two theorems about the P versus NP problem}, https://arxiv.org/pdf/1805.01755.pdf 
	
	\bibitem{wpd}
	https://en.wikipedia.org/wiki/G\"odel's\_incompleteness\_theorems
	
\end{thebibliography}
\end{document}